\documentclass[10pt,bibliography=totocnumbered
]{scrartcl}
\usepackage[utf8]{inputenc}
\usepackage[english]{babel}
\usepackage{aliascnt}
\usepackage{amssymb,amsmath,amsthm}
\usepackage{dsfont}
\usepackage[left=25mm, right=25mm, top=25mm, bottom=20mm]{geometry}
\usepackage{tikz}

\usepackage{hyperref} 

\allowdisplaybreaks[4] 

\newtheorem{prop}{Proposition}[section]

\newaliascnt{lemma}{prop} 
\newtheorem{lemma}[lemma]{Lemma}

\aliascntresetthe{lemma} 

\newaliascnt{mydef}{prop} 
 
\aliascntresetthe{mydef} 

\newaliascnt{corol}{prop} 
 
\aliascntresetthe{corol}

\newaliascnt{remark}{prop} 
 \newtheorem{remark}[remark]{Remark}
 
\aliascntresetthe{remark} 

\newaliascnt{mythm}{prop} 
 \newtheorem{mythm}[mythm]{Theorem}
 
\aliascntresetthe{mythm} 

\newaliascnt{example}{prop} 
 
\aliascntresetthe{example}

\def\equationautorefname~#1\null{%
  (#1)\null
}

\newcommand{\R}{\ensuremath{\mathbb{R}}}
\newcommand{\N}{\ensuremath{\mathbb{N}}}
\newcommand{\Z}{\ensuremath{\mathbb{Z}}}

\newcommand*\diff{\mathop{}\!\mathrm{d}}

\newcommand{\cL}{\mathcal{L}}
\newcommand{\cA}{\mathcal{A}}

\title{The Fractional Laplacian has Infinite Dimension}
\author{Adrian Spener\thanks{Institute of Applied Analysis, University of Ulm, Helmholtzstra\ss e 18, 89081 Ulm, Germany. 
\texttt{adrian.spener@uni-ulm.de}}, Frederic Weber\thanks{Institute of Applied Analysis, University of Ulm, Helmholtzstra\ss e 18, 89081 Ulm, Germany. \texttt{frederic.weber@uni-ulm.de}},
 and Rico Zacher\thanks{Institute of Applied Analysis, University of Ulm, Helmholtzstra\ss e 18, 89081 Ulm, Germany. 
\texttt{rico.zacher@uni-ulm.de}}}

\begin{document}

\maketitle

\begin{abstract}We show that the fractional Laplacian on $\R^d$ fails to satisfy the Bakry-\'Emery curvature-dimension inequality $CD(\kappa,N)$
for all curvature bounds $\kappa\in \R$ and all finite dimensions $N>0$. 
\end{abstract}


 




\section{Introduction}

The fractional Laplacian is an important example of a non-local operator and has received considerable
attention in recent years. It appears in many fields of analysis and probability theory, for example in the theory of stochastic processes as generator of a class of pure jump processes,
so-called $\alpha$-stable L\'evy processes \cite{MR2512800,MR1918242,MR2465826}. It can be further identified with a certain 
Dirichlet-to-Neumann operator, which is the key property of the so-called harmonic extension method as described by
Caffarelli and Silvestre in \cite{MR2354493}. During the last decade, non-local PDEs involving a fractional Laplacian or other related
non-local operators have been studied intensively, see e.g. \cite{MR2244602,MR2270163,MR2354493,MR2448308,MR2498561,MR2465826,MR2494809,MR2680400,MR2847534,MR3169755,MR3122168,MR3165278,MR3339179,MR3280032,2016arXiv160807571I}.

In the present paper we study the fractional Laplacian in $\R^d$ from a geometric analyst's point of view. More precisely,
we are interested in the question whether the fractional Laplacian satisfies a curvature-dimension (CD) inequality with {\em finite}
dimension. Curvature-dimension inequalities play a central role in the theory of Markov semigroups and operators and
related functional inequalities like e.g. the Poincar\'e or logarithmic Sobolev inequality (cf. \cite{MR3155209}). They also constitute an important
tool in geometric analysis as they encode information of the geometric properties (curvature and dimension) of the underlying
structure, e.g. a Riemannian manifold \cite{MR2962229}.\\
There exist several different notions of CD-inequalities. In this paper, we use the classical one, which was introduced by 
Bakry and \'Emery and is formulated in terms of the carr\'e du champ operator $\Gamma$ and its iterated operator $\Gamma_2$
associated with the infinitesimal generator $\cL$ of a Markov semigroup, see \cite{MR889476}. 
Another widely
used approach relies on the theory of optimal transport, see \cite{MR2459454}. 
For Riemannian manifolds, both approaches lead to the same notion of lower bounds for the Ricci curvature.  

To describe the notion of CD-inequality in the Bakry-\'Emery sense, let $\cL$ be the generator of a symmetric Markov semigroup with invariant reversible measure $\mu$ on the state space $E$. The bilinear operators $\Gamma$ and $\Gamma_2$ are defined by
\begin{align*}
\Gamma(u,v)& = \frac{1}{2}\left( \cL(uv) - u \cL v - v \cL u\right),\\
\Gamma_2(u,v)& = \frac{1}{2} \left( \cL \Gamma (u,v) - \Gamma (u, \cL v) - \Gamma (\cL u,v)\right)
\end{align*}
on a suitable algebra $\cA$ of real-valued functions $u,v$ defined on the underlying state space $E$. 
Furthermore, one sets
\begin{align}
 \Gamma(u) & := \Gamma(u,u) = \frac{1}{2} \cL  (u^2) - u \cL u,\label{eq:DefGamma}\\
 \Gamma_2(u)& := \Gamma_2(u,u) = \frac{1}{2} \cL \Gamma (u) - \Gamma (u, \cL u).\label{eq:DefGamma2}
\end{align}
 
Let $\kappa\in \R$ and $N \in (0, \infty]$. The operator $\cL$ is said to satisfy the \emph{Bakry-\'Emery curvature-dimension inequality} $CD(\kappa, N)$ with dimension $N$ and curvature $\kappa$ (lower bound) at $x\in E$ if 
\begin{equation} \label{eq:BE}
 \Gamma_2(u)(x) \geq \frac{1}{N} \big((\cL u)(x)\big)^2 + \kappa \Gamma(u)(x)
\end{equation}
for all functions $u:\,E\to \R$ in a sufficiently rich class $\cA$ of functions. We further say that $\cL$ satisfies the
$CD(\kappa,N)$-inequality, if \eqref{eq:BE} holds $\mu$-almost everywhere for all $u\in \cA$, cf.\ \cite[Sect. 1.16]{MR3155209}.\\
As an illustrating example we let $E = (M,g)$ be a Riemannian manifold with canonical Riemannian measure $\mu_g$ and $\mathcal{L} = \Delta_g$ the Laplace-Beltrami operator. Using the Bochner-Lichnerowicz formula one obtains that $CD(\kappa, N)$ is equivalent to $\operatorname{Ric}_g(x) \geq \kappa g(x)$ and $\dim M \leq N$.\\

In this paper, we take as operator $\cL$ the fractional Laplacian of order $\beta\in (0,2)$ on $E=\R^d$, which can be defined by
\begin{equation}                                                                                                                         
\label{eq:DefLaplace}
{\mathcal{L}}u(x) := -(-\Delta u)^\frac {\beta}{2}(x) = c_{\beta,d} \int_{\R^d} \frac{u(x+h) - 2u(x) +u(x-h)}{|h|^{d+\beta} } \diff h,\quad x\in \R^d,                                                                                                             \end{equation}
for functions $u$ which are sufficiently regular and do not grow too fast as $\lvert x \rvert\to \infty$, for instance functions in the Schwartz space $\mathcal{S}(\R^d)$. The normalising constant $c_{\beta, d}$ is given by 
\[
c_{\beta, d} := \frac{2^\beta \Gamma(\frac{d+\beta}{2})}{\pi^\frac{d}{2}|\Gamma(-\frac{\beta}{2})|},
\]
where $\Gamma$ now denotes the Gamma function.
We refer to \cite[Theorem 1.1. (e)]{10definitons}, which collects several other equivalent definitions of the fractional
Laplacian. Furthermore, the Lebesgue measure on $\R^d$ plays the role of the invariant reversible measure $\mu$.

Does the fractional Laplacian satisfy the $CD(\kappa,N)$-inequality with some $\kappa\in \R$ and {\em finite} $N>0$?
This question in the special case $\kappa=0$ has been recently identified as a key open problem in extending
the Gamma calculus of Bakry and \'Emery to non-local operators such as the fractional Laplacian. It has been posed
in the recent survey of Garofalo in \cite[(20.14)]{FractionalThoughts}, who also describes possible applications
such as {\em non-local Li-Yau inequalities}. We remark that the fractional Laplacian satisfies $CD(0,\infty)$ for all $\beta\in (0,2)$.
This is a trivial consequence of the representation formula for the associated $\Gamma_2$ operator, which implies that 
$\Gamma_2(u)\ge 0$ for all $u$, cf. formula \eqref{eq:Gamma2} below.
\\
Recall that in the case of the classical Laplacian $\Delta$ on $\R^d$ (which corresponds to the limiting case $\beta=2$) there holds
\[\Gamma_2(u) = \lvert\operatorname{Hess}(u)\rvert_{\operatorname{HS}}^2 \geq \frac{1}{d}( \Delta u)^2.\]
That is, the Laplacian satisfies the $CD(0,d)$-inequality. By means of this inequality and the chain rule for the Laplacian one can
show that positive solutions $u$ of the heat equation in $\R^d$ satisfy the celebrated Li-Yau gradient estimate or differential Harnack inequality (\cite{MR834612})
\[
|\nabla (\log u)(t,x)|^2-\partial_t (\log u)(t,x)\le \frac{d}{2t},\quad t>0,\,x\in \R^d.
\]
Integrating this estimate along suitable paths in space-time, one can deduce a sharp version of the parabolic Harnack
inequality. More generally, any Markov diffusion operator for which $CD(0,N)$ holds true with a finite $N > 0$ satisfies a Li-Yau inequality by \cite[Cor. 6.7.5]{MR3155209}.
\\

As the main result of this paper, we give a negative answer to Garofalo's question by proving there exists no finite $N> 0$ such that the fractional Laplacian satisfies $CD(0,N)$.
 As a corollary  to this result we obtain by means of a scaling argument that the situation is even worse, namely there exist no pair $(\kappa,N)\in \R\times (0,\infty)$ such that the fractional Laplacian satisfies the $CD(\kappa,{ {N}})$-condition. In particular, a negative curvature term does not save the CD-inequality with finite dimension. 
In this sense, the fractional Laplacian {\em has infinite dimension}. Our main result reads as follows. Note that here we 
refer to the representation formulas \eqref{eq:Gamma} and \eqref{eq:Gamma2} for $\Gamma(u,v)$ and $\Gamma_2(u)$,
respectively (see Lemma \ref{represent} below), which hold for sufficiently regular functions which do not grow too fast at infinity, in particular for $\mathcal{C}_c^\infty$-functions.

\begin{mythm} \label{thm:main}
Let $\beta\in (0,2)$, $d \in \N$ and $\mathcal{L}$ be the fractional Laplacian from \eqref{eq:DefLaplace} with associated operators $\Gamma$ and $\Gamma_2$ given below by \eqref{eq:Gamma} and \eqref{eq:Gamma2}, respectively.
Then
for any $R > 0$, any $\kappa \in \R$ and any finite $N > 0$ there exists some smooth and compactly supported function $u \in \mathcal{C}^\infty_c(\R^d)$ such that
\begin{equation}
 \label{eq:BEInequ} 
 \Gamma_2(u)(x) < \kappa \Gamma(u)(x) + \frac{1}{N} \left( \mathcal{L}(u)(x)\right)^2 
\end{equation}
for all $x \in \R^d$ with $\lvert x\rvert \leq R$. 
\end{mythm}

Theorem \ref{thm:main} follows from Theorem \ref{thm:FinalBall} and Proposition \ref{prop:NoCurvature}.\\

Our proof of the first part proceeds in several steps. We first consider the case $d=1$. Given a finite $N>0$, we construct an {\em unbounded} function $u$ for which $CD(0,N)$ fails at $x=0$. The key idea, which is motivated by \cite[Theorem 3.1]{allgemeinDiskret}, is to consider functions  which are smooth enough near $x=0$ and equal to $\lvert x\rvert^{\beta-\varepsilon}$ for $\lvert x\rvert\ge 1$ with small $\varepsilon>0$.
 In the very recent work \cite{allgemeinDiskret}, which deals with CD-inequalities for non-local discrete operators on the lattice $\Z$,
 the described family of functions is used to prove that $CD(0,N)$ fails for all finite $N>0$ for a certain class of operators with power type kernel,
 in particular for all powers $-(-\Delta_{disc})^{\frac{\beta}{2}}$ of the discrete Laplacian $\Delta_{disc}$ on $\Z$ with $\beta\in (0,2)$. 
 
 By considering functions which depend only on one real variable, the one-dimensional counterexample can be extended to
 the multi-dimensional case. To get a counterexample with compactly supported functions, we carry out an approximation argument which involves carefully chosen cut-off functions.
 The argument is very technical, since the estimates for the $\Gamma_2$ operator require distinguishing  several
 cases with respect to the domain of integration.  Finally, to see the failure of $CD(0,N)$ on an arbitrary ball around zero,
 we use a continuity argument to first prove the result for a sufficiently small ball and then use the scaling properties
 of the fractional Laplacian and its associated $\Gamma_2$ operator to get the desired result on any ball.

\subsubsection*{Acknowledgement} Adrian Spener and Rico Zacher are supported by the DFG (project number 355354916, GZ ZA 547/4-1).

\section{Gamma Calculus and the Fractional Laplacian}

Let $u$ be smooth enough and growing slowly enough such that the following integrals exist, for instance let $u$ be in the Schwartz space $\mathcal{S}(\R^d)$. Then the quadratic operators from the $\Gamma$-calculus are given as follows.
\begin{lemma} \label{represent}
Let $0 < \beta < 2$ and $\mathcal L$ be the fractional Laplacian from \eqref{eq:DefLaplace} with associated $\Gamma$ from \eqref{eq:DefGamma} and $\Gamma_2$ from \eqref{eq:DefGamma2}.
Then
\begin{equation} \label{eq:Gamma}
 \Gamma(u,v)(x) = {c_{\beta,d}}\int_{\R^d}\frac{(u(x+h)-u(x))(v(x+h)-v(x))}{|h|^{d+\beta}}\diff h 
\end{equation}
and 
 \begin{equation}
  \label{eq:Gamma2}
  \Gamma_2(u)(x) =  {c_{\beta,d}^2}\int_{\R^d}\int_{\R^d}\frac{[u(x+h+\sigma)-u(x+h)-u(x+\sigma)+u(x)]^2}{|h|^{d+\beta}|\sigma|^{d+\beta}}\diff h\diff \sigma.
 \end{equation}
\end{lemma}

\begin{proof}
 The first equality is given for instance in \cite[Lemma 20.2]{FractionalThoughts}, the second one can be obtained from a short calculation, where one uses $a^2 - b^2 - 2b(a-b) = (a-b)^2$ for
 $a = u(x+h+\sigma)-u(x+\sigma)$, $b = u(x+h)-u(x)$.
\end{proof}

This immediately implies that for all $0<\beta<2$ the fractional Laplacian satisfies $CD(0,\infty)$ at all $x \in \R^d$.

\begin{lemma}\label{lemma:WLOG}
For fixed $x \in \R^d$ in \eqref{eq:BEInequ} we may assume without loss of generality that $x=0$ and $u(0) = 0$.
\end{lemma}

\begin{proof}
 Replacing an arbitrary function $v$ with $u(y) := v({ {x-y}}) - v(x)$ we obtain a function $u$ satisfying the assertions of \autoref{lemma:WLOG} as well as $\mathcal{L}u(0) = \mathcal{L}v(x)$, $\Gamma u(0) = \Gamma v(x)$ and $\Gamma_2 u(0) = \Gamma_2v(x)$.
\end{proof}

Using a scaling argument, we may assume without loss of generality that $\kappa = 0$ in \eqref{eq:BEInequ}, i.e. to show that the fractional Laplacian has infinite dimension we may assume that $\kappa = 0$. This is the content of the next proposition.

\begin{prop} \label{prop:NoCurvature}
$CD(\kappa, N)$ with some finite $ N > 0$ and some $\kappa \in \R$ implies $CD(0,N)$.
\end{prop}

\begin{proof} By means of Lemma \ref{lemma:WLOG} we may assume that $x = 0$ and $u(0) = 0$.  The assertion is evident
in the case $\kappa > 0$. Now assume that $CD(-\kappa, N)$ holds for some $\kappa > 0$ and finite $N > 0$. Then we have
 \begin{equation}
  \label{eq:Gamma1wird0}
\frac{1}{N}(\mathcal{L}(u)(0))^2 \leq  \Gamma_2(u){ {(0)}} +\kappa \Gamma(u){ {(0)}}
 \end{equation}
 for all $u$ smooth and growing slowly enough. For $\lambda > 0$ we let $v(x) = u(\lambda x)$, and calculate
 \[
  \mathcal{L}(v)(0) = c_{\beta,d} \lambda^{d+\beta} \int_{\R^d} \frac{u(\lambda h) + u(-\lambda h)}{|\lambda h|^{d+\beta}} \diff h
  = c_{\beta,d} \lambda^{\beta} \int_{\R^d} \frac{u(x) + u(-x)}{\lvert x \rvert^{d+\beta}} \diff x
  = \lambda^\beta \mathcal{L}(u)(0).
 \]
Similarly we find
\[
 \Gamma(v)(0) = {c_{\beta,d}\lambda^{d+\beta}} \int_{\R^d} \frac{u(\lambda h)^2}{|\lambda h|^{d+\beta}}\diff h = \lambda^\beta \Gamma(u)(0)
\]
and
\[
\Gamma_2(v)(0) = {c_{\beta,d}^{ 2}\lambda^{2d+2\beta}} \int_{\R^d}\int_{\R^d}\frac{[u(\lambda h+ \lambda \sigma)-u(\lambda h)-u(\lambda\sigma)]^2}{|\lambda h|^{d+\beta}|\lambda \sigma|^{d+\beta}}\diff h\diff \sigma =
\lambda^{2\beta}\Gamma_2(u)(0).
\]
Whence, from \eqref{eq:Gamma1wird0} we find $
\frac{1}{N}(\mathcal{L}(v)(0))^2 \leq \Gamma_2(v)(0) + \lambda^\beta \kappa \Gamma(v)(0)$
for all $\lambda > 0$, which implies the assertion.
\end{proof}

Let us give a small lemma which will simplify the calculation in the third section.

\begin{lemma}\label{lem:simplified}
Let $d=1$ and $u$ be an even function with sufficient smoothness and growing slowly enough which satisfies $u(0) = 0$, then
\begin{equation}
 \label{eq:Gamma_2_simplified}
  \Gamma_2(u)(0) = 4c_{\beta,1}^2 \int_0^\infty \int_0^h \frac{[u(h+\sigma)- u(h) - u(\sigma)]^2+[u(h-\sigma)- u(h) - u(\sigma)]^2}{h^{{ {1}}+\beta} \sigma^{{ {1}}+\beta}} \diff \sigma \diff h.
\end{equation}
\end{lemma}

\begin{proof}
Denote the integrand by $A(h,\sigma) = \frac{[u(h+\sigma)- u(h) - u(\sigma)]^2}{h^{{ {1}}+\beta} \sigma^{{ {1}}+\beta}}$, then $A(h,\sigma) = A(-h, -\sigma)$ and $A(h, \sigma ) =A(\sigma ,h)$. The result follows from Fubini's Theorem.
\end{proof}

 \begin{lemma} \label{lemma:Reduction}
  {Let $u \colon \R^d \to \R$ be smooth and growing slowly enough and assume that $u$ depends only on one variable, i.e. $u(x_1, \ldots, x_d) = v(x_1)$ for some $v\colon \R \to \R$. Denote the fractional Laplacian on $\R^d$ temporarily by $\mathcal{L}^{(d)}$, similarly for the corresponding $\Gamma$-operators. 
  Then}
  \[
 \mathcal{L}^{(d)}u(0) = A_{d,\beta}\mathcal{L}^{(1)}(v)(0),\quad  \Gamma^{(d)}(u)(0) = A_{d,\beta}\Gamma^{(1)}(v)(0), \quad \text{and} \quad \Gamma_2^{(d)}(u)(0) = A_{d,\beta}^2 \Gamma^{(1)}_2(v)(0),
  \]
  where $A_{d,\beta} > 0$ is a constant.
 \end{lemma}

 \begin{proof}
  We will use that for $x\neq0$, $2\alpha > 1$ one obtains
     \[
   \int_{\R}\frac{1}{(x{ {^2}} + y^2)^{\alpha}} \diff y 
   = \frac{1}{x^{{{ {2}}\alpha}}}\int_{\R}\frac{1}{(1+ (\frac{y}{{ {x}}})^2)^{\alpha}} \diff y 
   =\frac{1}{x^{ {2\alpha -1}}}\int_{\R}\frac{1}{(1+ z^2)^{\alpha}} \diff z 
   =  \frac{B_\alpha}{x^{ {2\alpha -1}}},
  \]
  where $B_\alpha = \int_{\R}{(1+ z^2)^{-\alpha}} \diff z$
  is independent of $x$.
  Thus, one inductively finds for $2 \alpha > n$ that
  \[
   \int_{\R^{n}}\frac{1}{(x^2 + y_1^2 + \ldots +y_n^2)^{\alpha}} \diff y = \frac{C_{\alpha,n}}{|x|^{2\alpha - n}}
  \]
  with $C_{\alpha,n} = \Pi_{j=0}^{n-1} B_{\alpha - \frac{j}{2}}$.

Whence, we find
\begin{align*}
 \mathcal{L}^{(d)}(u)(0) 
 &= c_{\beta,d}\int_{\R}\left(v(h_1)-2v(0) + v(-h_1) \right)\int_{\R^{d-1}}\frac{1}{(h_1^2 + h_2^2 + \ldots + h_d^2)^{\frac{d+\beta}{2}}}\diff (h_2,\ldots,h_d) \diff h_1 \\
 &=c_{\beta,d}C_{\frac{d+\beta}{2},d-1} \int_\R \frac{v(h_1)-2v(0) + v(-h_1)}{|h_1|^{1+\beta}} \diff h_1\\
 &= \frac{ c_{\beta,d} C_{\frac{d+\beta}{2},d-1}}{c_{\beta,1}} \mathcal{L}^{(1)}(v)(0)
 = A_{d,\beta} \mathcal{L}^{(1)}(v)(0)
\end{align*}
with $A_{d,\beta} = \frac{ c_{\beta,d} C_{\frac{d+\beta}{2},d-1}}{c_{\beta,1}}$. The formulas for $\Gamma^{(d)}(u)$ and $\Gamma_2^{(d)}(u)$ follow analogously.
 \end{proof}

\section{Proof of the Main Result}

 In this section we gradually construct our counterexample for \eqref{eq:BE}. First, we construct an unbounded but admissible smooth function with unbounded support such that \eqref{eq:BE} fails at $x=0$. With an elaborate approximation scheme we subsequently construct a smooth and compactly supported function such that \eqref{eq:BE} fails at $x=0$, which will be used to show \eqref{eq:BEInequ} using a continuity and a rescaling argument.

\subsection{The Unbounded Counterexample with Unbounded Support}
 In this subsection we construct a class of unbounded functions with noncompact support lying in the domain of both $\mathcal{L}$ and $\Gamma_2$ such that \eqref{eq:BE} fails. 

\begin{mythm}\label{thm:MainTheoremExact} Let $0<\beta < 2$, and $u_\varepsilon \in \mathcal{C}^1(\R,\R)$ be an even function defined by
 \[
  u_\varepsilon(x) = \begin{cases}
          \lvert x \rvert^{\beta - \varepsilon}, & \lvert x \rvert\geq 1,\\
          \phi(x), & \lvert x \rvert<1,
         \end{cases}
 \]
where $0<\varepsilon < \frac{\beta}{2}$ (and additionally $\varepsilon < \frac{\beta -1}{2}$ if $\beta > 1$), and assume there exist $\Lambda>0,\delta > 0$ independent of $\varepsilon$ with $\beta + \delta > 1$ such that $\phi$ satisfies 
\[ 0 \leq \phi(x) \leq \Lambda x^{\beta + \delta},  \qquad |\phi'(x)|\leq \Lambda x^{\beta + \delta-1}\quad \text{ for }\lvert x \rvert\leq 1.
\]
Then there exist constants $C_0>0, C_1 > 0$ such that
\[
 \mathcal{L}(u_\varepsilon)(0)  \geq  \frac{C_0}{\varepsilon}, \qquad 0<\Gamma_2(u_\varepsilon)(0)  \leq\frac{ C_1 }{\varepsilon}.
 \]
In particular, for all $\mu > 0$ there exists some $u$ such that 
 $ 0 < \Gamma_2(u)(0) <\mu \left(\mathcal{L}(u)(0)\right)^2 < \infty$.
\end{mythm}
\begin{remark}\label{rem:cutoffzero}
 Let $u(x) = \lvert x \rvert^{\beta - \varepsilon} \eta(x)$, where $\eta \geq 0$ is a smooth and even cut-off function satisfying $\eta (x) = 0$ for $\lvert x \rvert \leq \frac{1}{4}$ and $\eta (x) = 1$ for $\lvert x \rvert \geq \frac{3}{4}$. Then $u$ satisfies the assumptions of Theorem \ref{thm:MainTheoremExact}.
 \end{remark}
 
 We may extend the previous counterexample to higher spatial dimensions.

\begin{mythm}\label{thm:higherdimension}
 Let $d \geq 2$ and $v(x_1, \ldots, x_d) = u(x_1)$, where $u=u_\varepsilon$ is given as in Theorem \ref{thm:MainTheoremExact}.
 Then there exist constants $C_2>0, C_3 > 0$ such that
\[
 \mathcal{L}(v)(0) \geq \frac{C_2}{\varepsilon}, \qquad 0<\Gamma_2(v)(0) \leq\frac{ C_3 }{\varepsilon}
\]
for all $\varepsilon > 0$ small enough.
In particular, for all $\mu > 0$ there exists some $v$ such that 
$ 0 < \Gamma_2(v)(0) <\mu \left(\mathcal{L}(v)(0)\right)^2 < \infty$.
\end{mythm}

\begin{proof}
 This is a direct consequence of Lemma \ref{lemma:Reduction} and Theorem \ref{thm:MainTheoremExact}.
\end{proof}

  For the proof of Theorem \ref{thm:MainTheoremExact} we need the following lemmata.
\begin{lemma}\label{lem:mvtforueps}
Let $u_\varepsilon$ be as in Theorem \ref{thm:MainTheoremExact} and $0 \leq x \leq y $. Then we have
\begin{align*}
\left(u_\varepsilon(y) - u_\varepsilon(x) \right)^2 \leq \begin{cases} \beta^2 \max\{ x^{2\beta -2 \varepsilon - 2} , y^{2\beta - 2 \varepsilon -2} \} (y-x)^2, & x \geq 1 \\
\Lambda^2y^{2\beta + 2\delta-2}(y-x)^2, &y \leq 1 \\
(\beta + \Lambda)^2\max\{1, y^{2\beta-2\varepsilon-2} \} (y-x)^2,& \text{else}.
\end{cases}
\end{align*}
\end{lemma}
\begin{proof}
This is a straightforward application of the mean value theorem together with the properties of $u_\varepsilon$.
\end{proof}

\begin{lemma}\label{lem:mixedbrackets}
\label{eq:klammermixedminus}
Let $u_\varepsilon$ be as in Theorem \ref{thm:MainTheoremExact} and $1 \leq \sigma \leq h $. Then
\begin{equation}\label{eq:klammermixedplus}
\left(u_\varepsilon(h\pm\sigma) - u_\varepsilon(h) - u_\varepsilon(\sigma) \right)^2
\leq \begin{cases}16 h^{2\beta-2\varepsilon -2} \sigma^2, &  \beta > 1, \\
4\sigma^{2\beta - 2\varepsilon}, & \beta \leq 1,
\end{cases}
\end{equation}
{ {where we additionally assume $h - \sigma \geq 1$ in the case of $`-`$.}}
\end{lemma}

\begin{proof}
Introduce $f,g:[0,1]\to \R$ by $f(x)=(1+x)^\gamma - 1 -x^\gamma$ and $g(x) = 1+ x ^\gamma - (1-x)^\gamma$, where $0 < \gamma < 2$.
Since $f(0) = 0$ we have for any $x \in [0,1]$  some $0 < \xi < x$ such that $f(x)= f'(\xi) x$ by the mean value theorem. Since for $\gamma \geq 1$ we have $f'(\xi) = \gamma \left( \left(1+\xi\right)^{\gamma-1} - \xi^{\gamma-1} \right) \in (0,4)$ we obtain 
\begin{equation}\label{eq:festimategeq1}
0 \leq f(x) \leq 4 x, \qquad x \in [0,1].
\end{equation}
The same argument shows for $\gamma \geq 1$ the bounds
\begin{equation}\label{eq:gestimategeq1}
0 \leq g(x) \leq 4 x, \qquad x \in [0,1].
\end{equation}
If $\gamma<1$, then $f'(\xi) < 0$ and $f(x) \geq - x^\gamma$, thus
\begin{equation}\label{eq:festimateleq1}
- x^\gamma \leq f(x) \leq 0, \qquad x \in [0,1].
\end{equation} 
Similarly $g(x) \geq x^\gamma \geq 0$, and for an upper bound of $g$ we consider $h(x) := x^{-\gamma} g(x), x > 0$. We can extend $h$ continuously by setting $h(0) = 1$. Since $h'(x) = \gamma x^{-1-\gamma}(1-x)^{\gamma -1} ( 1-(1-x)^{1-\gamma}) > 0$ for $0 < x < 1$ one obtains $h(x) \leq h(1) = 2$, whence
\begin{equation}\label{eq:gestimateleq1}
{ 0\leq {x^\gamma}} \leq g(x) \leq 2 x^\gamma, \qquad x \in [0,1].
\end{equation}
{ {We set $\gamma = \beta - \varepsilon$ and note that $\gamma > 1$ if $\beta > 1$ as well as $\gamma \leq 1$ if $\beta \leq 1$ by the bound of $\varepsilon$ in Theorem \ref{thm:MainTheoremExact}.}} We obtain from the definition of $u_\varepsilon$ that $(u_\varepsilon(h+\sigma) - u_\varepsilon(h) - u_\varepsilon(\sigma))^2 = h^{2\beta - 2\varepsilon} f\left(\frac{\sigma}{h}\right)^2$
 and $(u_\varepsilon(h-\sigma) - u_\varepsilon(h) - u_\varepsilon(\sigma))^2= h^{2\beta - 2\varepsilon} 
 g\left(\frac{\sigma}{h}\right)^2$, 
thus \eqref{eq:klammermixedplus} follows from 
\eqref{eq:festimategeq1}, \eqref{eq:gestimategeq1}, \eqref{eq:festimateleq1} and \eqref{eq:gestimateleq1}.
\end{proof}

We may now show Theorem \ref{thm:MainTheoremExact}.

\begin{proof}[Proof of Theorem \ref{thm:MainTheoremExact}]
 Fix some $\varepsilon_0>0$ sufficiently small and let $0 < \varepsilon < \varepsilon_0$. Throughout this proof we write $a(\varepsilon) \lesssim b(\varepsilon)$ whenever there exists some constant $C > 0$ independent of $\varepsilon$ such that $a(\varepsilon) \leq C b(\varepsilon)$.\\
 On the one hand we obtain due to symmetry
\begin{align*}
\mathcal{L}(u_\varepsilon)(0) 
= 2 { {c_{\beta,1}}} \int\limits_0^\infty \frac{u_\varepsilon (h)}{h^{1+\beta}} \diff h 
= 2 { {c_{\beta,1}}} \left( \int\limits_0^1 \frac{\phi(h)}{h^{1+\beta}} \diff h + \int\limits_1^\infty \frac{h^{\beta-\varepsilon}}{h^{1+\beta}} \diff h \right) 
\geq 2 { {c_{\beta,1}}}\int\limits_1^\infty \frac{h^{\beta-\varepsilon}}{h^{1+\beta}} \diff h = \frac{C_0}{\varepsilon}
\end{align*}
by our assumptions on $\phi$. Clearly, we have $\Gamma_2(u_\varepsilon) > 0$. To show the desired upper bound we will now split the integrals in $\Gamma_2$ into multiple parts. 
From Lemma \ref{lem:simplified} we find
\begin{align*}
&\frac{\Gamma_2(u_\varepsilon)(0)}{4c_{\beta,1}^2} = \int_0^\infty \int_0^h \frac{[u_\varepsilon(h+\sigma)- u_\varepsilon(h) - u_\varepsilon(\sigma)]^2+[u_\varepsilon(h-\sigma)- u_\varepsilon(h) - u_\varepsilon(\sigma)]^2}{h^{d+\beta} \sigma^{d+\beta}} \diff \sigma \diff h \\
&\quad = {\int\limits_0^1 \int\limits_0^h \frac{[u_\varepsilon(h+\sigma)- u_\varepsilon(h) - u_\varepsilon(\sigma)]^2}{h^{1+\beta}\sigma^{1+\beta}}\diff \sigma \diff h} + {\int\limits_0^1 \int\limits_0^h \frac{[u_\varepsilon(h-\sigma)- u_\varepsilon(h) - u_\varepsilon(\sigma)]^2}{h^{1+\beta}\sigma^{1+\beta}} \diff \sigma \diff h} \\
&\qquad + {\int\limits_1^\infty \int\limits_0^1 \frac{[u_\varepsilon(h+\sigma)- u_\varepsilon(h) - u_\varepsilon(\sigma)]^2}{h^{1+\beta}\sigma^{1+\beta}}\diff \sigma \diff h} + {\int\limits_1^\infty \int\limits_0^1 \frac{[u_\varepsilon(h-\sigma)- u_\varepsilon(h) - u_\varepsilon(\sigma)]^2}{h^{1+\beta}\sigma^{1+\beta}} \diff \sigma \diff h} \\
&\qquad + {\int\limits_1^\infty \int\limits_1^h \frac{[u_\varepsilon(h+\sigma)- u_\varepsilon(h) - u_\varepsilon(\sigma)]^2}{h^{1+\beta}\sigma^{1+\beta}}\diff \sigma \diff h} + {\int\limits_1^\infty \int\limits_1^h \frac{[u_\varepsilon(h-\sigma)- u_\varepsilon(h) - u_\varepsilon(\sigma)]^2}{h^{1+\beta}\sigma^{1+\beta}} \diff \sigma \diff h}
\\&\quad =: A_+ + A_- + B_+ + B_- + C_+ + C_-.
\end{align*}
We will show that $A_++A_-+B_++B_- \lesssim 1$ and $C_+ + C_- \lesssim  \frac{1}{\varepsilon}$. In most of the cases we make use of the basic estimate 
\[(u_\varepsilon(h\pm\sigma)-u_\varepsilon(h)-u_\varepsilon(\sigma))^2 \lesssim u_\varepsilon(\sigma)^2 + ( u_\varepsilon(h\pm\sigma) - u_\varepsilon(h))^2.\] 
Moreover, since $u_\varepsilon(\sigma)^2 = \phi(\sigma)^2$ for $\sigma \leq 1$ we obtain\begin{align*}
\int\limits_0^1 \frac{1}{h^{1+\beta}}\int\limits_0^h \frac{\phi(\sigma)^2}{\sigma^{1+\beta}} \diff \sigma \diff h \lesssim \int\limits_0^1 \frac{1}{h^{1+\beta}} \int\limits_0^h \sigma^{\beta+2\delta -1} \diff \sigma \diff h \lesssim \int\limits_0^1 \frac{1}{h^{1- 2 \delta}} \diff h \lesssim 1.
\end{align*}
Hence, for the integrals  $A_+$ and $A_-$ we merely have to estimate the corresponding expressions involving the term $\left( u(h\pm\sigma) - u(h)\right)^2$. The same observation holds for the integrals $B_+$ and $B_-$, since
\begin{align*}
\int\limits_1^\infty \frac{1}{h^{1+\beta}} \int\limits_0^1 \frac{\phi(\sigma)^2}{\sigma^{1+\beta}} \diff \sigma \diff h \lesssim \int\limits_1^\infty \frac{1}{h^{1+\beta}} \int\limits_0^1 \sigma^{\beta+ 2\delta -1} \diff \sigma \lesssim 1.
\end{align*}
\underline{$A_+$:} We split the remaining integral into 
\begin{align*}
\int\limits_0^\frac{1}{2}\int\limits_\sigma^{1-\sigma}\frac{\left(u_\varepsilon(h+\sigma) - u_\varepsilon(h))\right)^2}{\sigma^{1+\beta} h^{1+\beta}} \diff h \diff \sigma + \int\limits_{\frac{1}{2}}^1 \int\limits_{1-h}^h  \frac{\left(u_\varepsilon(h+\sigma)-u_\varepsilon(h)\right)^2}{h^{1+\beta}\sigma^{1+\beta}} \diff \sigma \diff h.
\end{align*}
In the first integral we have $h+\sigma \leq 1$, thus we obtain $\left(u_\varepsilon(h+\sigma) - u_\varepsilon(h)\right)^2 \lesssim \sigma^2 (h+\sigma)^{2\beta + 2\delta -2} \lesssim \sigma^2 h^{2\beta +2 \delta -2}$ from Lemma \ref{lem:mvtforueps}. Whence
\begin{align*}
\int\limits_0^\frac{1}{2}\int\limits_\sigma^{1-\sigma} \frac{\left( u_\varepsilon(h+\sigma)-u_\varepsilon(h)\right)^2}{\sigma^{1+\beta}h^{1+\beta}}\diff h \diff \sigma &\lesssim \int\limits_0^\frac{1}{2} \sigma^{1-\beta} \int\limits_\sigma^{1-\sigma} h^{\beta + 2\delta -3} \diff h \diff \sigma \lesssim 1
\end{align*}
since the last integral can be estimated by
$\int_0^\frac{1}{2} \sigma ^{1-\beta} \diff \sigma \lesssim 1$ for $\beta + 2\delta - 2>0$, by $\int_0^\frac{1}{2} \log (\frac{1}{\sigma})\sigma ^{1-\beta} \diff \sigma \lesssim 1$ for $\beta + 2\delta - 2=0$  and if $\beta + 2 \delta - 2 < 0$ we have the upper bound
$\int_0^\frac{1}{2} \sigma ^{-1 +2\delta} \diff \sigma \lesssim 1$.\\
In the other case of $h+\sigma \geq 1$, we have
$(u_\varepsilon(h+\sigma) - u_\varepsilon(h))^2 \lesssim \max\{1, h^{2\beta-2\varepsilon-2} \} \sigma^2$ by Lemma \ref{lem:mvtforueps}, thus
\begin{align*}
\int\limits_{\frac{1}{2}}^1 \int\limits_{1-h}^h \!\frac{\left(u_\varepsilon(h+\sigma)-u_\varepsilon(h)\right)^2}{h^{1+\beta}\sigma^{1+\beta}} \diff \sigma \diff h \lesssim \int\limits_{\frac{1}{2}}^1 \frac{\max \{1,h^{2\beta - 2\varepsilon -2}\}}{h^{1+\beta}}\!\! \int\limits_{1-h}^h \frac{\sigma^2}{\sigma^{1+\beta}} \diff \sigma \diff h
\lesssim
\int\limits_{\frac{1}{2}}^1 \!h^{1-2\beta} {\max \{1,h^{2\beta - 2\varepsilon -2}\}} \diff h,
\end{align*}
which can be bounded from above by some constant independent of $\varepsilon$.\newline
\underline{$A_-$:} 
According to Lemma \ref{lem:mvtforueps} we can estimate the remainder by
\begin{align*}
\int\limits_0^1 \int\limits_0^h \frac{\left(u_\varepsilon(h) - u_\varepsilon(h-\sigma) \right)^2}{h^{1+\beta}\sigma^{1+\beta}} \diff \sigma \diff h \lesssim \int\limits_0^1 h^{\beta + 2\delta -3} \int\limits_0^h \sigma^{1-\beta} \diff \sigma \diff h \lesssim \int\limits_0^1 \frac{1}{h^{1-2\delta}} \diff h \lesssim 1.
\end{align*}
\underline{$B_+$:} We apply Lemma \ref{lem:mvtforueps} together with $h \leq h+\sigma \leq 2h$ to obtain 
\begin{align*}
\int\limits_1^\infty \int\limits_0^1 \frac{\left( u_\varepsilon(h+\sigma) - u_\varepsilon(h)\right)^2}{h^{1+\beta}\sigma^{1+\beta}}\diff \sigma \diff h \lesssim \int\limits_1^\infty \frac{h^{2\beta - 2\varepsilon - 2}}{h^{1+\beta}} \int\limits_0^1 \sigma^{1-\beta} \diff \sigma \lesssim \int\limits_1^\infty \frac{1}{h^{3-\beta+2\varepsilon}}\diff h \leq \frac{1}{2-\beta}.
\end{align*}
\underline{$B_-$:} This part requires the distinction between the cases when $h-\sigma \geq 1$ and $h-\sigma < 1$. First, we consider the situation where $h-\sigma < 1$ holds. Then we have $h \in [1,2]$, and
by Lemma \ref{lem:mvtforueps} we find
\begin{align*}
\int\limits_{h-1}^1 \frac{\left( u_\varepsilon(h)- u_\varepsilon(h - \sigma)\right)^2}{\sigma^{1+\beta}} \diff \sigma \lesssim \max\{1, h^{2\beta-2\varepsilon -2}\} \int\limits_{h-1}^1 \sigma^{1-\beta} \diff \sigma  \lesssim \max\{1, h^{2\beta-2\varepsilon -2} \}.
\end{align*}
If $2 \beta - 2\varepsilon-2 > 0$, we can calculate $\displaystyle\int_1^2 \frac{h^{2\beta- 2\varepsilon-2}}{h^{1+\beta}}\diff h \leq \frac{1}{2-\beta+2\varepsilon} \leq \frac{1}{2-\beta}$, while the case of $2 \beta - 2\varepsilon-2 \leq 0$ is clear.
Thus
\begin{align*}
\int\limits_1^2 \frac{1}{h^{1+\beta}}\int\limits_{h-1}^1 \frac{\left( u_\varepsilon(h)- u_\varepsilon(h - \sigma)\right)^2}{\sigma^{1+\beta}} \diff \sigma \diff h \lesssim 1.
\end{align*}
We will subdivide the case of $h-\sigma \geq 1$ in two integrals. Since 
$ \displaystyle  \int_0^{h-1} \frac{\sigma^2}{\sigma^{1+\beta}}\diff \sigma \lesssim (h-1)^{2-\beta}$ and $\displaystyle\int_0^1 \frac{\sigma^2}{\sigma^{1+\beta}} \diff \sigma \lesssim 1 ,
$
we can deduce from Lemma \ref{lem:mvtforueps} that
\begin{align*}
\int\limits_1^2\int\limits_0^{h-1} \frac{\left(u_\varepsilon(h)-u_\varepsilon(h-\sigma)\right)^2}{h^{1+\beta}\sigma^{1+\beta}}\diff \sigma \diff h  \lesssim 1 \quad  \text{ and }  \quad \int\limits_2^\infty\int\limits_0^{1} \frac{\left(u_\varepsilon(h)-u_\varepsilon(h-\sigma)\right)^2}{h^{1+\beta}\sigma^{1+\beta}}\diff \sigma \diff h  \lesssim 1.
\end{align*}
\underline{$C_+$:} Due to Lemma \ref{lem:mixedbrackets} we distinguish two cases. If $\beta > 1$ one has
$\left(u_\varepsilon(h+\sigma)-u_\varepsilon(h) - u_\varepsilon(\sigma)\right)^2 \lesssim {h^{2\beta-2\varepsilon-2}}\sigma^2$, whence
\[C_+ \lesssim 
 \int\limits_1^\infty \frac{h^{2\beta-2\varepsilon-2}}{h^{1+\beta}}\int\limits_1^h \frac{\sigma^2}{\sigma^{1+\beta}}\diff \sigma \diff h 
 \lesssim \int\limits_1^\infty \frac{1}{h^{3-\beta + 2 \varepsilon}} {\int\limits_1^h \sigma^{1-\beta} \diff \sigma}
\diff h \lesssim \int\limits_1^\infty \frac{1}{h^{1+2\varepsilon}} \diff h = \frac{1}{2\varepsilon}.
\]
If $\beta \leq 1$, we deduce
\[C_+ \lesssim \int\limits_1^\infty \frac{1}{h^{1+\beta}}\int\limits_1^h \frac{\sigma^{2\beta - 2\varepsilon}}{\sigma^{1+\beta}} \diff \sigma \diff h 
\lesssim \int\limits_1^\infty \frac{1}{h^{1+ \beta}} {\int\limits_1^h \sigma^{\beta - 2\varepsilon -1} \diff \sigma} \diff h \lesssim \int\limits_1^\infty \frac{1}{h^{1+2\varepsilon}} \diff h = \frac{1}{2\varepsilon}.
\]
\underline{$C_-$:} If $h-\sigma < 1$ we have the upper bound 
\begin{align*}
\int\limits_1^\infty \frac{1}{\sigma^{1+\beta}}\int\limits_\sigma^{\sigma + 1} \frac{1}{h^{1+\beta}} h^{2\beta - 2 \varepsilon} \diff h \diff \sigma \lesssim \int\limits_1^\infty \frac{\sigma^{\beta- 2 \varepsilon -1}}{\sigma^{1+\beta}} \diff \sigma = \int\limits_1^\infty \frac{1}{\sigma^{2+2\varepsilon}}\diff \sigma \leq \frac{1}{1+2\varepsilon} \leq 1,
\end{align*}
since  $\phi(h-\sigma)^2 + u_\varepsilon(\sigma)^2 \lesssim \sigma^{2\beta - 2 \varepsilon} \leq h^{2\beta - 2\varepsilon}=u_\varepsilon(h)^2$ holds.
If $h-\sigma \geq 1$ holds, we again employ Lemma \ref{lem:mixedbrackets} to deduce
$C_- \lesssim \frac{1}{2\varepsilon}$ analogously to the estimate of $C_+$.
We conclude $\Gamma_2 (u_\varepsilon)(0) \lesssim \frac{1}{\varepsilon}$, which shows the claim.
 \end{proof}

\subsection{Compactly Supported and Smooth Counterexamples}
The aim of this subsection is to give a counterexample from the class of smooth, compactly supported functions. 

\begin{mythm}\label{thm:cc}
Let $0 < \beta < 2$ and $d \in \N$. Then for all $\mu > 0$ there exists some $u \in \mathcal{C}_c^\infty(\R^d)$ satisfying 
\begin{equation}\label{eq:CDtestfunctions}
0 < \Gamma_2(u)(0) < \mu \left(\mathcal L { {(u)}}(0) \right)^2.
\end{equation}Moreover, $u(x) = 0$ for all $\lvert x \rvert \leq
\frac{1}{4}$.
\end{mythm}

\begin{proof}
{ 
Let us first show how to reduce Theorem \ref{thm:cc} to the case of $d=1$.
Assume that for given $\mu > 0$ there exists a $v \in \mathcal{C}_c^\infty(\R)$ with $v(x_1) = 0$ for all $|x_1|\leq \frac{1}{4}$ and $0 < \Gamma_2(v)(0) < \frac{\mu}{4}\left(\mathcal L { {(v)}}(0) \right)^2$. We will now show \eqref{eq:CDtestfunctions} for some $u \in \mathcal{C}_c^\infty(\R^d)$ satisfying $u(x) = 0$ for all $\lvert x \rvert \leq
\frac{1}{4}$.\\
We set $w(x_1, \ldots, x_d) := v(x_1)$, then $w \in \mathcal{C}^\infty(\R^d) \cap W^{1, \infty}(\R^d)$. Furthermore, we define $u_n(x) :=  \chi_n(x) w(x) \in \mathcal{C}_c^\infty(\R^d)$, where $\chi_n \in \mathcal{C}_c^\infty(\R^d)$ is a 
cutoff function satisfying $0 \leq  \chi_n \leq 1$, $\chi_n \equiv 1$ on $B_n(0)$, $\chi_n \equiv 0$ on $\R^d \setminus B_{n+1}(0)$ and $\lvert \nabla \chi_n\rvert_\infty \leq C$ for all $n \in \N$.
Note that $u_n(x) = 0$ for all $|x| \leq \frac{1}{4}$ and $(u_n)_n$ is uniformly bounded in $W^{1,\infty}(\R^d)$ by, say, $M>0$. We will now show that 
\begin{equation}
 \label{eq:limitCutoff}
  \mathcal{L}(u_n)(0) \to \mathcal{L}(w)(0)\quad \text{ and }\quad\Gamma_2(u_n)(0) \to \Gamma_2(w)(0)\quad \text{ as }n\to \infty
\end{equation}
using Lebesgue's theorem. Clearly $u_n \to w$ pointwise. For $|h|\leq \frac{1}{4}$ we have $u_n(h) 
= 0$, and for $|h|\geq \frac{1}{4}$ we can bound $\frac{|u_n(h) + u_n(-h)|}{|h|^{d+\beta}} \leq \frac{2 |u_n|_{\infty} }{|h|^{d+\beta}}\leq \frac{2 M}{|h|^{d+\beta}}$, which is integrable on $\R^d\setminus B_{\frac{1}{4}}(0)$. The first part of \eqref{eq:limitCutoff} follows. For the latter we define $ g_n(h,\sigma) := \frac{[u_n(h+\sigma) - u_n(h) - u_n(\sigma)]^2}{|h|^{d+\beta}|\sigma|^{d+\beta}}$ and $B := B_{\frac{1}{8}}(0) \subset \R^d$, then 
\begin{equation*} 
 c_{\beta,d}^{-2}\Gamma_2(u_n)(x)
  = \int_{\R^d\setminus B}\int_{\R^d\setminus B} g_n(h,\sigma) \diff \sigma \diff h
   + 2 \int_{\R^d\setminus B}\int_{B} g_n(h,\sigma) \diff \sigma \diff h.
\end{equation*}
In the first integral 
we estimate $g_n$ by the integrable function $9 M ^2 |h|^{-d-\beta}|\sigma|^{-d-\beta}$. In the second integral we use $u_n(\sigma)=0$ and apply the mean value theorem to obtain the upper bound
 $g_n(h,\sigma) \leq  |\nabla u_n|_\infty^2 \sigma^{-d +2 -\beta} h^{-d-\beta}\leq  M^2 \sigma^{-d +2 -\beta} h^{-d-\beta}$, which is integrable on $(\R^d \setminus B) \times B$ since $\beta <2$. Another application of Lebesgue's theorem yields the second part of \eqref{eq:limitCutoff}.
 \\
We can now infer the claim for $d>1$ using \eqref{eq:limitCutoff}.
Pick $n$ large enough such that $u_n \in \mathcal{C}^\infty_c(\R ^d)$ satisfies
$|(\mathcal{L}(u_n)(0))^2 - (\mathcal{L}(w)(0))^2| \leq \frac{1}{2} (\mathcal{L}(w)(0))^2$ and 
$|\Gamma_2(u_n)(0) - \Gamma_2(w)(0)| \leq \Gamma_2(w)(0)$.
Then $(\mathcal{L} (w)(0))^2 \leq 2 (\mathcal{L}(u_n)(0))^2$ and whence, using Lemma \ref{lemma:Reduction} and our choice of $v$,
\[
0<\Gamma_2(u_n)(0) \leq 2 \Gamma_2(w)(0)  =2 A_{d,\beta}^2 \Gamma_2(v)(0) < 2 A_{d,\beta}^2\frac{\mu}{4} (\mathcal{L} (v)(0))^2  = \frac{\mu}{2} (\mathcal{L} (w)(0))^2
\leq \mu (\mathcal{L} (u_n)(0))^2.
\]

In the remainder we prove the claim for $d=1$.} Let $u_\varepsilon$ be given as in Theorem \ref{thm:MainTheoremExact} with a smooth cutoff at zero as explained in Remark \ref{rem:cutoffzero}, and define $v_{N,\varepsilon}$ by $v_{N,\varepsilon} := u_\varepsilon \eta_N$, where $\eta_N$ is a symmetric and smooth cutoff function satisfying $\eta_N \equiv 1$ on $[0,N]$, $0 < \eta_N < 1$ on $(N,N^2)$, $\eta_N \equiv 0$ on $[N^2,\infty)$ and a decay behaviour of $\vert \eta_N' \vert \lesssim N^{-2}$. Note that we use here and throughout this proof the symbol `$\lesssim$' whenever the respective constant is independent of $N$. Moreover, $N>0$ is not related to the dimension in the CD-inequality in this subsection.\\
We have $v_{N,\varepsilon} \in \mathcal{C}^\infty_c(\R)$ with the desired vanishing near zero, and $v_{N,\varepsilon}'(x)=0$ for $x \geq N^2$ and
\begin{equation}\label{eq:derivativevN}
\vert v_{N,\varepsilon}'(x) \vert \leq (\beta - \varepsilon)\eta_N(x) x^{\beta
 - \varepsilon -1} + 
 \vert \eta_N'(x) \vert
 x^{\beta - \varepsilon} \lesssim x^{\beta - \varepsilon -1}, \quad x \in [1,N^2].
\end{equation}
Let $\beta \in(0,2)$ be fixed and let $0 < \varepsilon < \varepsilon_0$, where $0<4\varepsilon < \beta$ and additionally 
$4\varepsilon_0 < \beta -1$ (if $\beta > 1$) and $4\varepsilon_0 <\beta - \frac{1}{2}$ if $\beta > \frac{1}{2}$.
In the sequel we will denote $u_\varepsilon$ by $u$ and $v_{N,\varepsilon}$ by $v_N$.\\
First we obtain that 
\begin{align*}
| \mathcal{L}(u)(0) - \mathcal{L}(v_N)(0) | &\lesssim \left( \int\limits_N^{N^2} \frac{| u (h) - \eta_N(h) u(h)|}{h^{1+\beta}}\diff h + \int\limits_{N^2}^\infty \frac{u(h)}{h^{1+\beta}}\diff h \right) 
\lesssim \int\limits_N^\infty \frac{u(h)}{h^{1+\beta}} \diff h \to 0
\end{align*}
as $N \to \infty$. 
If also $\Gamma_2(v_N)(0) \to \Gamma_2 (u)(0)$ as $N \to \infty$ then the claim can be shown as { above, now using Theorem \ref{thm:MainTheoremExact}}. In the remainder of this proof we will thus show that $\Gamma_2(v_N)(0) \to \Gamma_2(u)(0)$.\\
First, we define for $M>0$ the kernel  $k_M(x) := \frac{1}{\lvert x \rvert^{1+\beta}}\mathds{1}_{[-M,M]}(x)$. We denote by $\mathcal{L}^M$ the corresponding operator
\begin{equation*}
\mathcal{L}^M u(x):= c_{\beta,1}\int_{[-M,M]} \frac{u(x+h) - 2u(x) +u(x-h)}{|h|^{d+\beta} } \diff h,                                                                                                             
\end{equation*} 
and by 
\[\Gamma_2^Mu(x) = c_{\beta,1}^2 \int_{[-M,M]}\int_{[-M,M]} \frac{[ u(x+h+\sigma) - u(x+h) -u(x+\sigma) +u(x)]^2}{|h|^{1+\beta} |\sigma|^{1+\beta} } \diff h \diff  \sigma
\]
 the corresponding iterated carr\'e du champ operator. 
Applying the theorem of dominated convergence one finds that $\Gamma_2^M (u)(0)$ converges to $\Gamma_2 (u)(0)$ as $M$ tends to infinity. Let $\rho > 0$ and fix $M>0$ such that 
\begin{equation*}
|\Gamma_2^M (u) (0) - \Gamma_2 (u) (0) | < \rho.
\end{equation*}
If $N > 2M$ and $|h|, |\sigma| \leq M$, then $( v_N(h+\sigma)-v_N(h) - v_N(\sigma))^2 = ( u(h+\sigma) - u(h) - u(\sigma))^2$ and thus $|\Gamma_2^M (v_N)(0) - \Gamma_2^M (u)(0)| = 0$. 
 This observation together with
\begin{align*}
| \Gamma_2 (v_N)(0) - \Gamma_2 (u)(0)| \leq | \Gamma_2 (v_N)(0) - \Gamma_2^M (v_N)(0)| + | \Gamma_2^M (v_N)(0) - \Gamma_2^M (u)(0) | + | \Gamma_2^M (u)(0) - \Gamma_2 (u) (0) |
\end{align*}
shows that it suffices to prove that $| \Gamma_2 (v_N)(0) - \Gamma_2^M (v_N)(0)|$ converges to zero as $N$ tends to infinity.
By symmetry (cf. Lemma \ref{lem:simplified}) we have  
\begin{equation}
 \label{eq:Gamma2triangle}
 \begin{split}
  &
  |\Gamma_2 (v_N)(0) - \Gamma_2^M (v_N)(0)| \\
& \quad \lesssim \int\limits_{M}^\infty \int\limits_0^h \frac{( v_N(h+\sigma) - v_N(h)-v_N(\sigma))^2 + ( v_N(h-\sigma) - v_N(h)-v_N(\sigma))^2}{h^{1+\beta}\sigma^{1+\beta}} \diff \sigma \diff h.
 \end{split}
\end{equation}
In order to show that this integral converges to zero we distinguish several cases (see Figures \ref{fig:domain1} and \ref{fig:domain2}). Starting with the `$h+\sigma$'-term we have the following: 

\begin{figure}[ht]
 \centering
 \begin{tikzpicture}[scale = 1]
 \draw (0,0) -- (7,7);
  \draw[->,thick] (0,-0.3) -- (0,8.4);
 \draw[->, thick] (-0.3,0) -- (11.4,0);
 \draw[densely dotted] (9,7) -- (9,0)-- (4.5,4.5);
 \draw[densely dotted] (3,0)-- (3,3);
 \draw
 (3,0)-- (1.5,1.5);
 \draw
 (1,0) -- (1,1);
 \node (j) at (-0.2, 6) {$\sigma$};
 \node (l) at (5, -0.2) {$h$};
 \node (n2c) at (1,-0.2) {M};
 \node (n3) at (3,-0.2) {N};
 \node (n4) at (9,-0.2) {N$^2$};
 \node (eins) at (1.5,0.5) {I};
 \node (zwei) at (5,2) {II};
   \node  at (-0.3,1) {M};
   \node  at (0,1) {-};
   \node  at (-0.3,3) {N};
    \node  at (0,3) {-};
\end{tikzpicture}
 \caption{Splitting of the domain of integration for the term with $h+\sigma$.}
  \label{fig:domain1}
\end{figure}
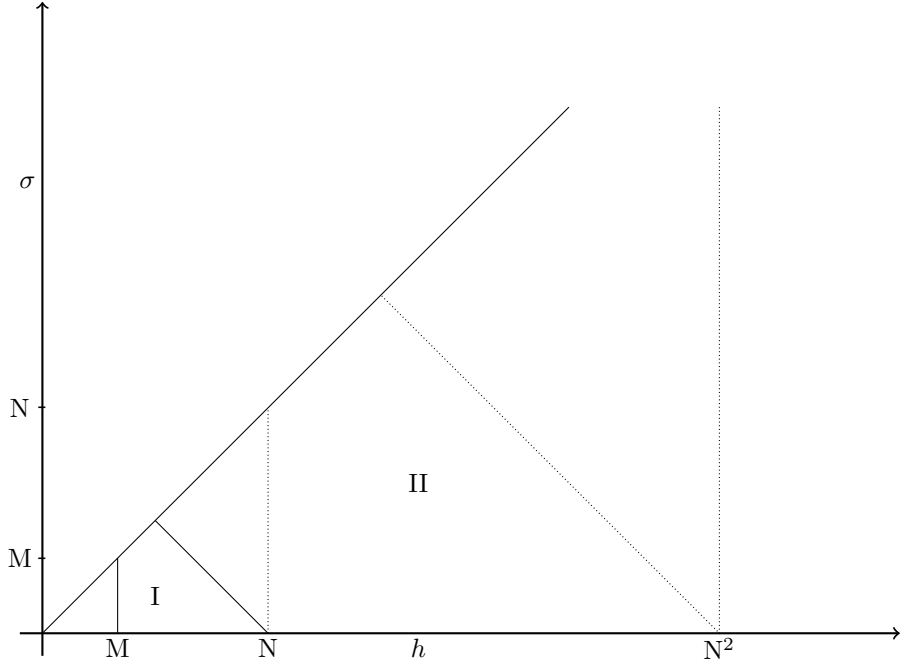

\underline{I: $0 \leq \sigma \leq h$, $h+\sigma \leq N$, $h \geq M$:} Clearly, since $h+\sigma \leq N$ one finds $( v_N(h+\sigma)-v_N(h) - v_N(\sigma) )^2 = ( u(h+\sigma) - u(h) - u(\sigma) )^2$ and thus we have the upper bound
\begin{equation*}
\int\limits_M^N \int\limits_0^{N-h} \frac{\left( v_N(h+\sigma) - v_N(h)-v_N(\sigma)\right)^2}{h^{1+\beta}\sigma^{1+\beta}} \diff \sigma \diff h \lesssim |\Gamma_2^M (u) (0) - \Gamma_2 (u) (0) | < \rho.
\end{equation*}

\underline{II: $0 \leq \sigma \leq h, h+\sigma \geq N$} As in the proof of Theorem \ref{thm:MainTheoremExact} we will make use of the basic estimate 
\begin{align*}
\left(v_N(h\pm\sigma)-v_N(h)-v_N(\sigma)\right)^2 \lesssim v_N(\sigma)^2 + \left( v_N(h\pm\sigma) - v_N(h)\right)^2
\end{align*} 
and note that for the $v_N(\sigma)$-term we have
\begin{align*}
\int\limits_\frac{N}{2}^\infty \frac{1}{h^{1+\beta}} \int\limits_0^1 \frac{v_N(\sigma)^2}{\sigma^{1+\beta}} \diff \sigma \diff h \lesssim \int\limits_\frac{N}{2}^\infty \frac{1}{h^{1+\beta}} \int\limits_0^1 \frac{\sigma^{2\beta + 2\delta}}{\sigma^{1+\beta}} \diff \sigma \diff h \lesssim \int\limits_\frac{N}{2}^\infty \frac{1}{h^{1+\beta}} \diff h \lesssim \frac{1}{N^\beta} 
\end{align*}
and 
\begin{align*}
\int\limits_\frac{N}{2}^\infty \frac{1}{h^{1+\beta}} \int\limits_1^h \frac{v_N(\sigma)^2}{\sigma^{1+\beta}} \diff \sigma \diff h \leq \int\limits_\frac{N}{2}^\infty \frac{1}{h^{1+\beta}} \int\limits_1^h \frac{\sigma^{2\beta - 2\varepsilon}}{\sigma^{1+\beta}} \diff \sigma \diff h &\lesssim  \int\limits_\frac{N}{2}^\infty \frac{1}{h^{1+2\varepsilon}} \diff h 
\lesssim \frac{1}{N^{2\varepsilon}}.
\end{align*}
These estimates will be applied also to the term involving `$h-\sigma$' later. 
From the mean value theorem we obtain $\left( v_N(h + \sigma ) - v_N(h) \right)^2 = v_N'(\xi)^2 \sigma^2 $ for some $h \leq \xi \leq h+\sigma \leq 2 h$. This yields by \eqref{eq:derivativevN} the upper bound
\begin{align*}
\int\limits_{\frac{N}{2}}^\infty \frac{1}{h^{1+\beta}} \int\limits_0^h \frac{\left( v_N(h + \sigma ) - v_N(h) \right)^2}{\sigma^{1+\beta}} \diff \sigma \diff h \lesssim \int\limits_{\frac{N}{2}}^\infty \frac{h^{2\beta - 2\varepsilon - 2}}{h^{1+\beta}} 
\int\limits_0^h \sigma^{1-\beta} \diff \sigma \diff h
&\lesssim \int\limits_{\frac{N}{2}}^\infty \frac{1}{h^{1 + 2\varepsilon}} \diff h \lesssim \frac{1}{N^{2 \varepsilon}}
\end{align*}
which tends to zero as $N \to \infty$. This shows that the first part of \eqref{eq:Gamma2triangle} tends to zero.\\
For the estimate of the second part we need a refined splitting of the domain, see Figure \ref{fig:domain2}. Note that we may omit the integrals containing $v_N(\sigma)^2$ by the previous calculation.

\begin{figure}[ht]
 \centering
\begin{tikzpicture}[scale = 0.55]
\footnotesize
 \draw (0,0) -- (12,12);
\draw[->,thick] (0,-0.3) -- (0,12.4);
 \draw[->, thick] (-0.3,0) -- (22.4,0);
 \draw[densely dotted] (0.5,0) -- (3,2.5);
 \draw
 (1,0) -- (1,1);
 \draw[densely dotted] (0.5,0.5)--(14,0.5);
 \draw[densely dotted] (9,0) -- (12,3);
 \draw[densely dotted] (3,0) -- (3.5,0.5);
   \draw[densely dotted] (3,3) -- (3.5,3);
   
 \draw (3.5,0) -- (3.5, 3);
 \draw (3,0)-- (3,3);
\draw (3,2.5) -- (12.5,12);
 \draw (3.5, 0.5) -- (6,3);
 \draw (3.5,3) -- (21,3);
 \draw (12,3) -- (18,9);
 \draw (9.5,9) -- (21,9);

 \node at (1.9,0.8) {\scriptsize A};
 \node (B) at (4.5,6) {\scriptsize B};
 \draw[->] (B) to [bend right](4.8,4.5); 
 \node at (3.25,2) {\scriptsize C};
 \node at (4.1,2){\scriptsize D};
 \node at (10,6) {\scriptsize E};
 \node at (19,6) {\scriptsize F};
 \node at (16,11) {\scriptsize G};
 \node at (12,1.8){\scriptsize H};

 \node  at (-0.4, 6) {$\sigma$};
 \node  at (14, -0.4) {$h$};
 \node  at (0.5,-0.4) {1};
 \node  at (1,-0.4) {M};
 \node  at (3,-0.4) {N};
 \node  at (9,-0.4) {N$^2$};
  \node (n3c) at (-0.5,9) {N$^2$};
   \node (n4c) at (-0.5,3) {N};
   \node (n3c) at (0,9) {-};
   \node (n4c) at (0,3) {-};
      \node  at (-0.5,0.5) {1};
   \node at (0,0.5) {-};
\end{tikzpicture}
 \caption{Splitting of the domain of integration for the term with $h-\sigma$.}\label{fig:domain2}
\end{figure}
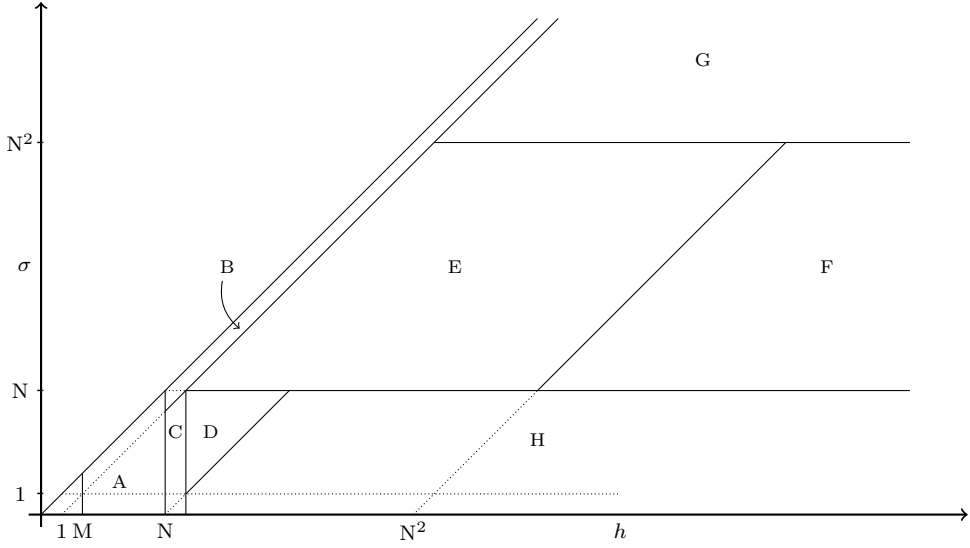
\underline{A: $0 \leq \sigma \leq h \leq N$, $h \geq M$:} 
Since $h - \sigma \leq N$ it follows similar to (I) that
\begin{equation*}
\int\limits_M^N \int\limits_{0}^h \frac{\left( v_N(h-\sigma) - v_N(h)-v_N(\sigma)\right)^2}{h^{1+\beta}\sigma^{1+\beta}} \diff \sigma \diff h < {\rho}.
\end{equation*}

\underline{B: $h-1 \leq \sigma \leq h$, $h \geq N$:} We can always estimate 
\begin{equation}
 \label{eq:estimateallvyu} 
\left( v_N(h-\sigma) - v_N(h)-v_N(\sigma)\right)^2 \leq 9 u(h)^2. 
\end{equation}
Hence, we deduce from $(h-1)^{-1-\beta} \lesssim h^{-1-\beta}$ that
\begin{align*}
&\int\limits_N^\infty \int\limits_{h-1}^h\frac{\left( v_N(h-\sigma) - v_N(h)-v_N(\sigma)\right)^2}{h^{1+\beta}\sigma^{1+\beta}} \diff \sigma \diff h \lesssim \int\limits_N^\infty \frac{h^{2\beta - 2\varepsilon}}{h^{1 + \beta}} \int\limits_{h-1}^h \frac{1}{\sigma^{1+\beta}} \diff \sigma
\\ & \qquad \lesssim  \int\limits_{N}^\infty \frac{h^{2\beta - 2\varepsilon}}{h^{2+2\beta}} \diff h\lesssim \frac{1}{N^{1+2\varepsilon}}.
\end{align*}
\underline{C, D, E, F, G, H: $0 \leq \sigma \leq h-1$, $h \geq N$, where $\beta > \frac{1}{2}$:}
The mean value theorem implies the existence of some $\xi$ with $1 \leq h-\sigma \leq \xi \leq h$ such that
$\left(v_N(h)-v_N(h-\sigma)\right)^2 = v'_N(\xi)^2\sigma^2$. 
We consider the case of $\beta >1$ first. Here we conclude 
$v_N'(\xi)^2 \lesssim h^{2\beta - 2\varepsilon - 2}$ from \eqref{eq:derivativevN}. Thus, we can estimate
\begin{align*}
\int\limits_N^\infty \frac{1}{h^{1+\beta}}\int\limits_0^{h-1} \frac{\left(v_N(h)-v_N(h-\sigma)\right)^2}{\sigma^{1+\beta}} \diff \sigma \diff h &\lesssim  \int\limits_N^\infty \frac{h^{2\beta - 2\varepsilon -2}}{h^{1+\beta}} \int_0^{h-1} \sigma^{1-\beta} \diff \sigma \diff h 
\lesssim \int\limits_N^\infty \frac{1}{h^{1+2\varepsilon}} \diff h
\lesssim \frac{1}{N^{2\varepsilon}},
\end{align*}
establishing the claim for $\beta > 1$. In the case of $\frac{1}{2}  < \beta \leq 1$ we obtain from $v_N'(\xi)^2 \lesssim (h-\sigma)^{2\beta - 2\varepsilon - 2}$ 
and $2 \beta - 2 \varepsilon -1 > 0$ the estimate
\begin{align*}
&\int\limits_N^\infty \frac{1}{h^{1+\beta}}\int\limits_0^{h-1} \frac{\left(v_N(h)-v_N(h-\sigma)\right)^2}{\sigma^{1+\beta}} \diff \sigma \diff h \lesssim \int\limits_N^\infty \frac{h^{1-\beta}}{h^{1+\beta}} \int\limits_0^{h-1} (h-\sigma)^{2\beta - 2\varepsilon -2} \diff \sigma \diff h \\ 
& \qquad=\int\limits_N^\infty \frac{1}{h^{2\beta}} \int\limits_1^h y^{2\beta - 2\varepsilon -2} \diff y \diff h \lesssim \int\limits_N^\infty \frac{1}{h^{1+2\varepsilon}} \diff h \lesssim \frac{1}{N^{2\varepsilon}}.
\end{align*}
Thus, the claim is established for any $\beta > \frac{1}{2}$. From now on we assume $\beta \leq \frac{1}{2}$.\\
\underline{C: $0 \leq \sigma \leq h-1$, $N \leq h \leq N+1$, where $\beta \leq \frac{1}{2}$:} Arguing as before we find \begin{align*}
\int\limits_N^{N+1} \frac{1}{h^{1+\beta}}\int\limits_0^{h-1} \frac{\left(v_N(h)-v_N(h-\sigma)\right)^2}{\sigma^{1+\beta}} \diff \sigma \diff h
\lesssim  \int\limits_N^{N+1} \frac{1}{h^{2\beta}} 
\int\limits_1^h y^{2\beta - 2\varepsilon -2} \diff y
\diff h \lesssim \int\limits_N^{N+1} \frac{1}{h^{2\beta}} \diff h \lesssim 
\frac{1}{N^{2\beta}}
\end{align*}
since $\int_1^h y^{2\beta - 2\varepsilon -2} \diff y$ is bounded, thus this term converges to zero.\\
\underline{D: $N+1 \leq h \leq 2N$, $\sigma \leq N$, $h-\sigma \leq N$, where $\beta \leq \frac{1}{2}$:}
Using \eqref{eq:estimateallvyu} we obtain the upper bound
\begin{align*}
\int\limits_{N+1}^{2N} \frac{u(h)^2}{h^{1+\beta}} \int\limits_{h-N}^{N} \frac{1}{\sigma^{1+\beta}} \diff \sigma \diff h &\lesssim N^{2\beta-2\varepsilon} \int\limits_{N+1}^{2N} \frac{(h-N)^{-\beta}}{h^{1+ \beta}} \diff h \leq \frac{N^{2\beta - 2\varepsilon}}{N^{1+ \beta}} \int\limits_{N+1}^{2N} (h-N)^{-\beta} \diff h
\\&= \frac{1}{N^{1+2\varepsilon -\beta}} \int\limits_1^N \frac{1}{y^\beta} \diff y  \lesssim \frac{1}{N^{2\varepsilon}}.
\end{align*}
\underline{E: $N \leq \sigma \leq N^2$, $1 \leq h-\sigma \leq N^2$, where $\beta \leq \frac{1}{2}$:} In order to estimate the integral
\begin{equation}\label{eq:hardestcase}
\int\limits_N^{N^2} \frac{1}{\sigma^{1+\beta}} \int\limits_{1+\sigma}^{N^2+\sigma} \frac{\left(v_N(h-\sigma) - v_N(h)-v_N(\sigma)\right)^2}{h^{1+\beta}}\diff h\diff \sigma
\end{equation}
we rewrite
\begin{align*}
&\left(v_N(h-\sigma)-v_N(h)-v_N(\sigma)\right)^2\\
& \quad = h^{2\beta - 2\varepsilon} \left[ \eta_N(h-\sigma) \left( 1 + \left( \frac{\sigma}{h}\right)^{\beta - \varepsilon} - \left( 1- \frac{\sigma}{h}\right)^{\beta - \varepsilon} \right)\right.  \\ 
&\qquad \quad  +  \left.\eta_N(h) -\eta_N(h-\sigma) + \left(\eta_N (\sigma) - \eta_N(h-\sigma)\right)\left(\frac{\sigma}{h}\right)^{\beta - \varepsilon} \right]^2 \\
&\quad \lesssim  h^{2\beta - 2\varepsilon} \eta_N(h-\sigma)^2 \left( 1 + \left( \frac{\sigma}{h}\right)^{\beta - \varepsilon} - \left( 1- \frac{\sigma}{h}\right)^{\beta - \varepsilon} \right)^2 \\
&\qquad \quad + h^{2\beta - 2\varepsilon} \left( \eta_N(h) - \eta_N(h-\sigma)\right)^2 + h^{2\beta - 2\varepsilon} \left( \eta_N(\sigma)- \eta_N(h - \sigma) \right)^2 \left( \frac{\sigma}{h}\right)^{2\beta - 2\varepsilon}.
\end{align*}
Now we want to find suitable estimates for these three terms. First, we obtain from Lemma \ref{lem:mixedbrackets}, with $\beta \leq \frac{1}{2}$, that $ h^{2\beta - 2\varepsilon} {\eta_N(h-\sigma)^2} \left( 1 + \left( \frac{\sigma}{h}\right)^{\beta - \varepsilon} - \left( 1- \frac{\sigma}{h}\right)^{\beta - \varepsilon} \right)^2 \lesssim \sigma^{2\beta - 2\varepsilon} = u(\sigma)^2$, and therefore we can control this term by the previous calculations in (II). The same holds true for the third term, as $h^{2\beta - 2\varepsilon} {\left( \eta_N(\sigma)- \eta_N(h - \sigma) \right)^2} \left( \frac{\sigma}{h}\right)^{2\beta - 2\varepsilon} \leq \sigma^{2\beta - 2\varepsilon}$. The remaining term can be estimated using the mean value theorem and our assumptions on $\eta_N$ by $h^{2\beta - 2\varepsilon} \left( \eta_N(h) - \eta_N(h-\sigma)\right)^2 \lesssim \frac{h^{2\beta - 2\varepsilon}\sigma^2}{N^4}$. Finally, we obtain 
\begin{align*}
\frac{1}{N^4}\int\limits_N^{N^2} \frac{1}{\sigma^{1+\beta}} \int\limits_{1+\sigma}^{N^2+\sigma} \frac{h^{2\beta - 2\varepsilon}\sigma^2}{h^{1+\beta}}\diff h\diff \sigma 
&\lesssim \frac{N^{2\beta - 4\varepsilon}}{N^4} \int\limits_N^{N^2} \sigma^{1-\beta} \diff \sigma \lesssim \frac{1}{N^{4\varepsilon}}
\end{align*}
and conclude that \eqref{eq:hardestcase} converges to $0$ as $N$ tends to infinity.\\
\underline{F: $N \leq \sigma \leq N^2$, $N^2 \leq h-\sigma$, where $\beta \leq \frac{1}{2}$:}  Here, $v_N(h-\sigma) = v_N(h) = 0$ and hence this integral converges to zero by the previous calculations in (II).
\\\underline{G: $N^2 \leq \sigma \leq h-1$, where $\beta \leq \frac{1}{2}$:}
Since $v_N(x) \leq N^{2\beta - 2\varepsilon}$ for all $x$ we have
\begin{align*}
&\int\limits_{N^2 + 1}^\infty \frac{1}{h^{1+\beta}} \int\limits_{N^2}^h \frac{\left(v_N(h-\sigma)-v_N(h)-v_N(\sigma)\right)^2}{\sigma^{1+\beta}} \diff \sigma \diff h = \int\limits_{N^2 + 1}^\infty \frac{1}{h^{1+\beta}} \int\limits_{N^2}^h \frac{v_N(h-\sigma)^2}{\sigma^{1+\beta}} \diff \sigma \diff h \\
&\qquad \leq N^{4\beta - 4\varepsilon} \int\limits_{N^2 + 1}^\infty \frac{1}{h^{1+\beta}} \int\limits_{N^2}^h \frac{1}{\sigma^{1+\beta}} \diff \sigma \diff h
\lesssim \frac{1}{N^{4\varepsilon}}.
\end{align*}
\underline{H: $0 \leq \sigma \leq N, h \geq N+1, h-\sigma \geq N$, where $\beta \leq \frac{1}{2}$:} 
In the remaining case we have  $N \leq h- \sigma \leq \xi$ and thus
we can estimate
\[
\int\limits_0^N \frac{1}{\sigma^{1+\beta}} \int\limits_{N+\sigma}^\infty \frac{v_N'(\xi)^2 \sigma^2}{h^{1+\beta}} \diff h \diff \sigma \lesssim N^{2\beta - 2\varepsilon - 2} \int\limits_0^N \sigma^{1-\beta} {\int\limits_{N+\sigma}^\infty \frac{1}{h^{1+\beta}} \diff h \diff \sigma}
\lesssim \frac{1}{N^{2\varepsilon}}.\qedhere
\]
\end{proof}

\subsection{Failure of \texorpdfstring{$CD(0,N)$}{CD(0,N)} on Arbitrary Balls}

In this subsection we finish the proof of Theorem \ref{thm:main}.

\begin{prop} \label{prop:locallyfalse}
 Let $0 < \beta < 2$, $d \in \N$. Let $\mu > 0$. Then there exists some $\rho > 0$ and a smooth and compactly supported function $u \in \mathcal{C}^\infty_c(\R^d)$ such that for all $x \in \R^d$ with $\lvert x \rvert \leq \rho$ one has
 \[
  0 < \Gamma_2(u)(x) < \mu (\mathcal{L}(u)(x))^2.
 \]
\end{prop}

\begin{proof}
 Let $\mu > 0$ and $u \in \mathcal{C}^\infty_c(\R^{ {d}})$ { {be}} as in Theorem \ref{thm:cc} satisfying $0 < \Gamma_2(u)(0) < \frac{\mu}{4} \left( \mathcal{L}(u)(0)\right)^2$ and $u(y) = 0$ for all $|y| \leq \frac{1}{4}$. We will first show continuity of $\mathcal{L}(u)$ and $\Gamma_2(u)$ at zero, i.e.
\begin{equation}
 \label{eq:continuity}
 \mathcal{L}(u)(x) \to \mathcal{L}(u)(0) \quad \text{ and } \quad \Gamma_2(u)(x) \to \Gamma_2(u)(0) \quad \text{ as } x \to 0.
 \end{equation}
 Since $u$ is continuous we can show \eqref{eq:continuity} by finding integrable functions dominating the integrands and applying Lebesgue's theorem, { {resembling the proof of Theorem \ref{thm:cc}}}. For the first part we observe
\[
 \frac{|(u(x+h)-2u(x) + u(x-h)|}{|h|^{{ d}+\beta}} \leq 4|u|_\infty \frac{1-\mathds{1}_{ {  B_{ \frac{1}{8}}(0)}}(h)}{|h|^{{ d}+\beta}}, \quad \lvert x \rvert\leq \frac{1}{8},
\]
and the latter is integrable on $\R^{ d}$, which shows $\mathcal{L}(u)(x) \to \mathcal{L}(u)(0)$ as $x \to 0$. The continuity of $\Gamma_2(u)$ at zero requires a more detailed analysis{ {, as already indicated in the proof of Theorem \ref{thm:cc}}}. We define 
\[
 g(x,h,\sigma) := \frac{[u(x+h+\sigma) - u(x+h) - u(x+\sigma) + u(x)]^2}{|h|^{{ d}+\beta}|\sigma|^{{ d}+\beta}},
\]
and $B := B_{\frac{1}{16}}(0) { \subset \R^d}$, then 
\begin{equation} \label{eq:Gamma_2splitted}
 c_{\beta,{ d}}^{-2}\Gamma_2(u)(x)
  = \int_B \int_B  g(x,h,\sigma) \diff \sigma \diff h 
   + \int_{\R^{ d}\setminus B}\int_{\R^{ d}\setminus B} g(x,h,\sigma) \diff \sigma \diff h
   + 2 \int_{\R^{ d}\setminus B}\int_{B} g(x,h,\sigma) \diff \sigma \diff h.
\end{equation}
Let $\lvert x \rvert \leq \frac{1}{8}$, then the first integral of 
\eqref{eq:Gamma_2splitted} vanishes. In the second integral of \eqref{eq:Gamma_2splitted} we estimate $g$ by the integrable function $16 |u|_\infty ^2 |h|^{-{ d}-\beta}|\sigma|^{-{ d}-\beta}$. In the third part of \eqref{eq:Gamma_2splitted} we apply the mean value theorem to obtain the upper bound
 $g(x,h,\sigma) \leq 4 |{ {\nabla u}}|^2_\infty \sigma^{{ -d+2}-\beta} h^{-{ d}-\beta}$, which is integrable on $(\R^{ d} \setminus B) \times B$. Whence we can apply Lebesgue's theorem to obtain \eqref{eq:continuity}.
\\
Having established \eqref{eq:continuity} we can show the claim by choosing $\rho > 0$ small enough such that 
\[|\Gamma_2u(x) - \Gamma_2(u)(0)| \leq \Gamma_2(u)(0)\quad \text{ and } \quad |\left(\mathcal{L}(u)(x)\right)^2-\left(\mathcal{L}(u)(0)\right)^2| \leq \frac{\left(\mathcal{L}(u)(0)\right)^2}{2}
 \quad \text{ for all } \lvert x \rvert\leq \rho.
 \]
 Then $\left(\mathcal{L}(u)(x)\right)^2 \geq \left(\mathcal{L}(u)(0)\right)^2 - |\left(\mathcal{L}(u)(0)\right)^2 - \left(\mathcal{L}(u)(x)\right)^2| \geq 
\frac{\left(\mathcal{L}(u)(0)\right)^2}{2}$, whence
\[
 \Gamma_2(u)(x) \leq 2 \Gamma_2(u)(0) < \frac{\mu}{2}\left(\mathcal{L}(u)(0)\right)^2 \leq \mu \left(\mathcal{L}(u)(x)\right)^2, \qquad \lvert x \rvert\leq \rho. \qedhere
\]
\end{proof}

Applying a scaling argument we can give a counterexample on arbitrary balls, which is our main result.

\begin{mythm} \label{thm:FinalBall}
 Let $0 < \beta < 2$, $d \in \N$ and $R > 0$. Then \eqref{eq:BE} with $\kappa =0$ fails on all of $\overline{B_R(0)} \subset \R^d$, i.e. for all $\mu > 0$ there exists some smooth and compactly supported function $u \in \mathcal{C}^\infty_c(\R^d)$ such that for all $x \in \R^d$ with $\lvert x \rvert \leq R$ one has
 \[
  0 < \Gamma_2(u)(x) < \mu (\mathcal{L}(u)(x))^2.
 \]
\end{mythm}

\begin{proof}
Let $R > 0$ and $\mu$ be given, and fix $\rho > 0$ and $u$ as in Proposition \ref{prop:locallyfalse}. Define $M := \frac{R}{\rho}$ and $v(x) := u\left( \frac{x}{M} \right)$. Then $v \in \mathcal{C}^\infty_c(\R^d)$ satisfies 
$\mathcal{L}(v)(x) = M^{-\beta} \mathcal{L}(u)\left( \frac{x}{M} \right)$ and $\Gamma_2(v)(x) = M^{-2\beta}\Gamma_2u\left( \frac{x}{M} \right)$, as a short calculation similar to the proof of Proposition \ref{prop:NoCurvature} shows. Thus we conclude
$ 0 < \Gamma_2(v)(x) < \mu \left(\mathcal{L}(v)(x)\right)^2$ for all $\left|\frac{x}{M}\right| \leq \rho$, i.e. $\lvert x \rvert \leq R$.
\end{proof}

\bibliographystyle{alpha}

\end{document}